\documentclass[12pt]{article}

\usepackage{amsmath, amssymb, amsfonts, amsthm}
\usepackage{mathtools}
\usepackage[dvipsnames]{xcolor}
\usepackage{graphicx}
\usepackage{geometry}
\usepackage[hypertexnames=false,colorlinks=true,linkcolor=blue,citecolor=blue,urlcolor=blue]{hyperref}
\newtheorem{theorem}{Theorem}[section]
\newtheorem{proposition}{Proposition}[section]
\newtheorem{remark}{Remark}[section]

\title{\textbf{A Minimal Bifurcation Model of Load Imbalance\\
in a Softmax Mixture-of-Experts Router}}

\author{O.~M. Kiselev\\
\textit{Innopolis University, Innopolis, Russia}\\
\texttt{o.kiselev@innopolis.ru}}
\date{}

\begin{document}

\maketitle

\begin{abstract}
We propose a minimal dynamical model of adaptive softmax routing for a two-expert Mixture-of-Experts (MoE) layer. The model is obtained as a mean-field limit of a discrete reinforcement rule: the selected expert receives a small score increment, while all scores undergo regularizing decay. In the symmetric case the limiting system has a supercritical pitchfork bifurcation: for weak feedback there is a unique stable balanced state, whereas above a critical feedback strength two stable asymmetric states appear. When an external asymmetry is added, the pitchfork unfolds into a pair of fold bifurcations forming a cusp in the control-parameter plane. We derive exact parametric equations for the bifurcation set and the local normal form of the cusp catastrophe. Numerical experiments connect this picture to empirical expert load, a small trainable MoE model, hard top-$1$ PyTorch routing, and a small classification experiment on \texttt{digits}. The results provide a controlled low-dimensional mechanism for abrupt transitions to load imbalance in adaptive MoE routers.
\end{abstract}

\noindent\textbf{Keywords:} Mixture-of-Experts, softmax routing, pitchfork bifurcation, fold bifurcation, cusp catastrophe, hysteresis, load imbalance.

\section{Introduction}

Mixture-of-Experts (MoE) architectures use sparse routing: for each input, only a small subset of experts is activated. This makes it possible to increase the number of parameters without a proportional increase in computational cost. The same mechanism, however, is sensitive to load imbalance: the router may start selecting a small group of experts systematically, leaving the remaining experts almost inactive.

The full training dynamics of an MoE model involves router parameters, expert parameters, the data distribution, discrete top-$k$ selection, additional load-balancing penalties, and optimization noise. A direct rigorous analysis of realistic MoE training in full generality is therefore unlikely to be tractable. This paper considers a narrower question: can one construct a minimal dynamical system in which softmax routing and positive feedback alone already produce multistability, hysteresis, and an abrupt transition to asymmetric load?

The main object of the paper is a one-dimensional reduction of an adaptive router. It gives a mathematically controlled model of a possible local mechanism of load collapse. Unlike a purely phenomenological normal form, the nonlinearity here comes from the exact two-expert softmax formula, and the ordinary differential equation itself is obtained as the limit of an explicit stochastic update rule.

\section{Related Work and Motivation}

The classical mixture of local experts architecture was introduced by Jacobs, Jordan, Nowlan, and Hinton \cite{Jacobs1991}; its hierarchical extension and EM interpretation were developed by Jordan and Jacobs \cite{JordanJacobs1994}. The idea of describing expert specialization as a phase transition predates modern large sparse MoE models. Kang and Oh \cite{KangOh1996} carried out a statistical-mechanical analysis of the generalization behavior of the classical mixture of experts. In their setting, for a small number of training examples the system is in a symmetric phase in which expert roles are not specialized; after crossing a critical point, a symmetry-broken phase appears in which the gating network effectively partitions the input space among experts. For hierarchical MoE models, the same work finds multiple successive phase transitions. This line of ideas--criticality, symmetry breaking, and expert specialization--is the closest predecessor of the bifurcation perspective developed here.

Modern sparse MoE models make the routing problem more acute; a recent survey of MoE methods for large language models is given in \cite{Cai2025Survey}. In sparsely-gated MoE, Switch Transformer, GLaM, and ST-MoE \cite{Shazeer2017,Fedus2022,Du2022,Zoph2022}, scaling is achieved by sparse expert activation, but stable training requires dedicated load-balancing mechanisms. If the router concentrates token flow on a small number of experts too early, routing collapse or degraded expert utilization may occur. Practical architectures therefore include auxiliary losses, router z-losses, capacity constraints, noise in routing scores, and other stabilizing components.

A closely related engineering line is represented by GShard \cite{Lepikhin2021}, where conditional computation is used in very large MoE models and load balancing is a necessary component of stable training and distributed execution. These works make clear that load balancing is a central architectural issue.

It is also important to distinguish approaches in which balancing is not only an added loss term. BASE Layers \cite{Lewis2021} formulate routing as a balanced assignment of tokens to experts; Expert Choice Routing \cite{Zhou2022} reverses the direction of selection by allowing experts to choose tokens; Loss-Free Balancing \cite{Wang2024} introduces expert-wise biases before top-$K$ routing and updates these biases according to recent expert load. The latter construction is close in spirit to the variables $b_i$ and $r_i$ in our model: balancing acts not merely as a static regularizer, but as an additional slow dynamics that influences routing scores.

More recent preprints study routing dynamics explicitly. Mouzouni \cite{Mouzouni2026} models routing as a congestion game and argues that expert load in open MoE checkpoints passes through several phases: early balancing, stabilized specialization, and late relaxation of balance in favor of quality. Rastegar \cite{Rastegar2026} analyzes the singular soft-to-hard limit of softmax-routed MoE and shows that small temperature concentrates complexity near routing interfaces; in a two-expert Gaussian calculation, a local symmetry-breaking mechanism appears.

Nevertheless, the language and tools of bifurcation theory are rarely used explicitly in MoE analysis. This paper occupies an intermediate position between the statistical-mechanical picture of phase transitions and the engineering literature on load balancing. We introduce a minimal adaptive system in which symmetry breaking, folds, a cusp, and hysteresis can be computed directly. The goal is to provide a local normal form for one possible mechanism of load imbalance.

The paper introduces an explicit stochastic model for the adaptation of router scores. The softmax nonlinearity is not postulated at the level of a normal form; it enters before reduction. For the two-expert reduction we obtain exact equations for the fold set and the hysteresis width in terms of feedback strength, temperature, forgetting, and skew. We also show how standard balancing mechanisms can be interpreted as moving the system out of the multistable region by decreasing the effective positive feedback. Numerical experiments test the mean-field predictions in batch-routing dynamics and illustrate which elements of the picture persist, and which change, in small hard-routed PyTorch MoE models.

\section{An Adaptive Softmax Router}

\subsection{Two-Expert Model}

Consider two experts with internal attractiveness scores $r_1^n$ and $r_2^n$ at a discrete step $n$. These scores are aggregate slow variables of the router. Their interpretation is simple: the selected expert receives positive reinforcement, while the absence of sustained reinforcement is compensated by score decay.

Selection probabilities are given by softmax with temperature $T>0$:
\begin{equation}
    p_i(r_1^n,r_2^n)
    =
    \frac{\exp(r_i^n/T)}
    {\exp(r_1^n/T)+\exp(r_2^n/T)},
    \qquad i=1,2.
    \label{eq:softmax-probabilities}
\end{equation}
Let $I_n\in\{1,2\}$ denote the expert selected at step $n$, with
\[
    \mathbb{P}(I_n=i\mid r_1^n,r_2^n)=p_i(r_1^n,r_2^n).
\]
We define the update rule
\begin{equation}
    r_i^{n+1}
    =
    r_i^n
    +
    \eta
    \left(
        a\mathbf{1}_{\{I_n=i\}}
        -
        \gamma r_i^n
        +
        b_i
    \right),
    \qquad i=1,2.
    \label{eq:discrete-update}
\end{equation}
Here $\eta>0$ is a small adaptation step, $a>0$ is the reinforcement strength for the selected expert, $\gamma>0$ is the forgetting coefficient, and $b_i$ is an external drift associated with data or architectural asymmetry.

\begin{proposition}[Mean-Field Limit]
Let $r^\eta(t)$ be the piecewise-linear interpolation of \eqref{eq:discrete-update} on the time scale $t=n\eta$. Then, on every finite time interval, as $\eta\to 0$ the process $r^\eta(t)$ converges in probability to the solution of
\begin{equation}
    \dot r_i = a p_i - \gamma r_i + b_i,
    \qquad i=1,2.
    \label{eq:ri-dynamics}
\end{equation}
\end{proposition}

\begin{proof}
This is a standard application of the ODE method for stochastic approximation \cite{EthierKurtz1986,Borkar2008}. For completeness we give the short verification in the present finite-dimensional setting. The conditional mean increment is
\[
    \mathbb{E}
    \left[
        \frac{r_i^{n+1}-r_i^n}{\eta}
        \,\middle|\,
        r_1^n,r_2^n
    \right]
    =
    a p_i(r_1^n,r_2^n)-\gamma r_i^n+b_i.
\]
The remaining part of the increment is a bounded martingale difference. Over a time interval of order one the number of steps is of order $\eta^{-1}$, hence the quadratic variation of the accumulated martingale contribution is of order $\eta$ and tends to zero. Since the right-hand side of the limiting system is smooth and Lipschitz on bounded sets, the standard Gronwall estimate gives convergence of the interpolations to the solution of \eqref{eq:ri-dynamics}.
\end{proof}

We analyze the limiting system \eqref{eq:ri-dynamics}. Introduce the score difference
\begin{equation}
    y = r_1-r_2.
\end{equation}
For the two-expert softmax we have the exact identity
\begin{equation}
    p_1(y)-p_2(y)
    =
    \tanh\frac{y}{2T}.
    \label{eq:softmax-difference}
\end{equation}
Subtracting the equations for $r_1$ and $r_2$ gives the closed scalar system
\begin{equation}
    \dot y
    =
    F(y;a,\gamma,T,h)
    =
    a\tanh\frac{y}{2T}
    -
    \gamma y
    +
    h,
    \qquad
    h=b_1-b_2.
    \label{eq:main-model}
\end{equation}
The load difference in this model is
\begin{equation}
    u(y)=p_1(y)-p_2(y)=\tanh\frac{y}{2T}.
\end{equation}
Thus asymmetric equilibria in $y$ correspond directly to asymmetric expert load.

System \eqref{eq:main-model} is a minimal model. Its role is to isolate one mechanism: positive feedback through softmax can compete with regularizing forgetting and produce multiple stable routing regimes.

\subsection{Linearization for \texorpdfstring{$N$}{N} Experts}

The full analysis below is for two experts. We first indicate how the two-expert model relates to the many-expert case. Consider the direct $N$-expert generalization of the mean-field system:
\begin{equation}
    \dot r_i
    =
    a p_i(r)
    -
    \gamma r_i
    +
    b_i,
    \qquad
    p_i(r)
    =
    \frac{\exp(r_i/T)}
    {\sum_{j=1}^N\exp(r_j/T)}.
    \label{eq:n-expert-mean-field}
\end{equation}
In the symmetric case $b_i=0$ there is an equilibrium
\[
    r_1=\cdots=r_N=\frac{a}{N\gamma}.
\]
The softmax Jacobian at this point is
\[
    \frac{\partial p_i}{\partial r_j}
    =
    \frac{1}{NT}
    \left(
        \delta_{ij}-\frac{1}{N}
    \right).
\]
Therefore the linearization of \eqref{eq:n-expert-mean-field} has eigenvalue $-\gamma$ along the common-shift direction $(1,\ldots,1)$ and eigenvalue
\begin{equation}
    \lambda_{\mathrm{contr}}
    =
    \frac{a}{NT}
    -
    \gamma
    \label{eq:n-expert-contrast-eigenvalue}
\end{equation}
on the $(N-1)$-dimensional subspace of contrast modes
\[
    \sum_{i=1}^N v_i=0.
\]
Thus the uniform state loses linear stability at
\begin{equation}
    a=N\gamma T.
    \label{eq:n-expert-threshold}
\end{equation}
For $N=2$ this condition coincides with the threshold $a=2\gamma T$ analyzed below.

This linearization is not a substitute for a full many-expert analysis. After loss of stability one obtains a multidimensional mode-selection problem that depends on data asymmetries, initialization, capacity constraints, and top-$k$ mechanics. The two-expert model in this paper is interpreted as a local reduction to one dominant contrast mode, for example the competition between two experts or two clusters of experts. In this sense it describes one elementary channel of balance loss within the full geometry of an $N$-expert MoE.

\subsection{Gradient Structure}

System \eqref{eq:main-model} is a gradient system on the line:
\begin{equation}
    \dot y = -\frac{dV}{dy},
\end{equation}
where the potential is
\begin{equation}
    V(y)
    =
    \frac{\gamma}{2}y^2
    -
    2aT\log\cosh\frac{y}{2T}
    -
    hy.
    \label{eq:potential}
\end{equation}
Stable equilibria correspond to local minima of $V$, and unstable equilibria to local maxima. This gives a useful interpretation of hysteresis: as parameters vary, one of the potential minima may disappear, forcing the system to move to another minimum.

\section{The Symmetric Case}

We first consider $h=0$, when there is no external preference for either expert. The system becomes
\begin{equation}
    \dot y
    =
    a\tanh\frac{y}{2T}
    -
    \gamma y.
    \label{eq:symmetric}
\end{equation}
The right-hand side is odd, so $y=0$ is always an equilibrium. The local bifurcation terminology used below is standard for one-dimensional vector fields; see, for example, \cite{GuckenheimerHolmes1983,Kuznetsov2004}.

\begin{theorem}[Pitchfork Bifurcation]
Let $a,\gamma,T>0$ and $h=0$. For system \eqref{eq:symmetric} the following hold.
\begin{enumerate}
    \item If $a<2\gamma T$, then $y=0$ is the unique equilibrium and is asymptotically stable.
    \item If $a=2\gamma T$, then $y=0$ is the unique equilibrium; it is asymptotically stable but nonhyperbolic.
    \item If $a>2\gamma T$, then $y=0$ is unstable and there are exactly two nonzero stable equilibria $\pm y_\ast$.
\end{enumerate}
At $a=2\gamma T$ the system undergoes a supercritical pitchfork bifurcation.
\end{theorem}

\begin{proof}
Let
\[
    F_0(y)=a\tanh\frac{y}{2T}-\gamma y.
\]
If $a<2\gamma T$, then for $y>0$ we have $\tanh(y/(2T))<y/(2T)$, hence
\[
    F_0(y)
    <
    \left(\frac{a}{2T}-\gamma\right)y
    <
    0.
\]
For $y<0$, oddness gives $F_0(y)>0$. Hence there are no nonzero equilibria and trajectories point toward zero.

At $a=2\gamma T$ the same argument gives $F_0(y)<0$ for $y>0$ and $F_0(y)>0$ for $y<0$, but
\[
    F_0'(0)=\frac{a}{2T}-\gamma=0.
\]
Thus the equilibrium is stable but nonhyperbolic.

Now assume $a>2\gamma T$. Then $F_0'(0)>0$, so $y=0$ is unstable. For $y>0$, the function $F_0$ is initially positive and tends to $-\infty$ as $y\to+\infty$, with asymptotics $a-\gamma y$. Moreover,
\[
    F_0'(y)
    =
    \frac{a}{2T}\operatorname{sech}^2\frac{y}{2T}
    -
    \gamma,
\]
and $F_0'(y)$ is strictly decreasing for $y>0$. Therefore there is exactly one zero $y_\ast>0$ on the positive half-line, and at this zero $F_0'(y_\ast)<0$. By oddness there is a symmetric zero $-y_\ast$, also stable.

The expansion near $y=0$ is
\[
    F_0(y)
    =
    \left(\frac{a}{2T}-\gamma\right)y
    -
    \frac{a}{24T^3}y^3
    +
    O(y^5).
\]
The cubic coefficient is negative, which is the supercritical pitchfork case.
\end{proof}

The critical condition
\[
    a=2\gamma T
\]
has a simple interpretation. Increasing the positive feedback $a$ and decreasing the temperature $T$ strengthen the router's tendency to specialize. Increasing $\gamma$ strengthens stabilizing forgetting and supports the balanced state.

\section{Symmetry Breaking and Fold Bifurcations}

We now consider the general case $h\neq 0$. The parameter $h$ breaks the symmetry between the two experts. The pitchfork then unfolds into a pair of fold bifurcations.

Equilibria of \eqref{eq:main-model} are defined by
\begin{equation}
    F(y;a,\gamma,T,h)=0.
    \label{eq:equilibrium}
\end{equation}
Their stability is determined by the sign of
\begin{equation}
    F_y(y)
    =
    \frac{a}{2T}\operatorname{sech}^2\frac{y}{2T}
    -
    \gamma.
    \label{eq:linearization}
\end{equation}
An equilibrium is asymptotically stable if and only if $F_y(y)<0$.

\begin{theorem}[Bifurcation Set]
Fold bifurcations of \eqref{eq:main-model} form the following parametrized curve in the $(a,h)$ parameter plane:
\begin{equation}
    a(q)=2\gamma T\cosh^2 q,
    \qquad
    h(q)=2\gamma T\bigl(q-\sinh q\cosh q\bigr),
    \qquad
    q\in\mathbb{R}.
    \label{eq:fold-param}
\end{equation}
The point $q=0$ corresponds to the cusp
\[
    a=2\gamma T,
    \qquad
    h=0.
\]
For fixed $a>2\gamma T$ the system has three equilibria if and only if
\begin{equation}
    |h| < H(a),
    \label{eq:three-equilibria-condition}
\end{equation}
where
\begin{equation}
    H(a)
    =
    2\gamma T
    \left(
        \sinh q_a\cosh q_a - q_a
    \right),
    \qquad
    q_a
    =
    \operatorname{arcosh}
    \sqrt{\frac{a}{2\gamma T}}.
    \label{eq:H}
\end{equation}
Inside this region two equilibria are stable and one is unstable.
\end{theorem}

\begin{proof}
A fold occurs when
\[
    F(y;a,\gamma,T,h)=0,
    \qquad
    F_y(y;a,\gamma,T,h)=0.
\]
Set $q=y/(2T)$. From $F_y=0$ we obtain
\[
    \frac{a}{2T}\operatorname{sech}^2 q=\gamma,
    \qquad
    a=2\gamma T\cosh^2 q.
\]
Substituting this into $F=0$ gives
\[
    h
    =
    \gamma y-a\tanh q
    =
    2\gamma Tq
    -
    2\gamma T\cosh^2 q\,\tanh q
    =
    2\gamma T(q-\sinh q\cosh q).
\]
This proves \eqref{eq:fold-param}.

For fixed $a>2\gamma T$, the equation $F_y=0$ has two critical points $y=\pm 2Tq_a$, with $q_a$ given by \eqref{eq:H}. Between the two values of $h$ at which the graph of $F$ is tangent to the axis, the equation $F=0$ has three roots; outside this interval it has one root. Substituting $q=\pm q_a$ into \eqref{eq:fold-param} gives the boundaries $h=\mp H(a)$.

Since the system is one-dimensional, stability alternates. When there are three roots, the outer equilibria are stable and the middle one is unstable.
\end{proof}

\section{Local Normal Form of the Cusp Catastrophe}

Near the point
\[
    y=0,
    \qquad
    a=2\gamma T,
    \qquad
    h=0
\]
the right-hand side has the expansion
\begin{equation}
    F(y;a,\gamma,T,h)
    =
    \left(\frac{a}{2T}-\gamma\right)y
    -
    \frac{a}{24T^3}y^3
    +
    h
    +
    O(y^5).
    \label{eq:normal-expansion}
\end{equation}
If
\[
    \mu=\frac{a}{2T}-\gamma,
    \qquad
    \varepsilon=h,
\]
then, up to smooth rescaling of the state variable and time, the local normal form is
\begin{equation}
    \dot x
    =
    \mu x
    -
    x^3
    +
    \varepsilon
    +
    \text{higher-order terms}.
    \label{eq:cusp-normal-form}
\end{equation}
The unperturbed system $\varepsilon=0$ has a pitchfork bifurcation, and $\varepsilon$ gives its imperfect unfolding. For the normal form without higher-order terms, the bifurcation set is
\begin{equation}
    4\mu^3-27\varepsilon^2=0,
    \qquad
    \mu>0.
\end{equation}
This is the local algebraic model of the cusp and is consistent with the exact parametric description \eqref{eq:fold-param}.

\begin{proposition}[Nondegeneracy of the Cusp]
At $(y,a,h)=(0,2\gamma T,0)$, system \eqref{eq:main-model} has a nondegenerate cusp singularity in the sense of the local normal form.
\end{proposition}

\begin{proof}
At the specified point,
\[
    F=0,\qquad F_y=0,\qquad F_{yy}=0,
\]
and the third derivative is nonzero:
\[
    F_{yyy}(0;2\gamma T,\gamma,T,0)
    =
    -\frac{\gamma}{2T^2}
    \neq 0.
\]
Moreover,
\[
    F_h=1,
    \qquad
    F_{ya}=\frac{1}{2T}\neq 0.
\]
Thus the parameters $a$ and $h$ provide an independent two-parameter unfolding of the cubic degeneracy. This is precisely the local structure of the cusp catastrophe.
\end{proof}

\section{Hysteresis and Interpretation for Expert Load}

Let $a>2\gamma T$ be fixed. Then for $|h|<H(a)$ the system has two stable states. One corresponds to preference for the first expert, and the other to preference for the second. In terms of load,
\[
    u=\tanh\frac{y}{2T},
\]
these states give nonzero values of $u$ with opposite signs.

If $h$ varies slowly, a trajectory remains near the current stable equilibrium until that equilibrium reaches one of the folds. At $h=H(a)$ or $h=-H(a)$ the corresponding local minimum of the potential \eqref{eq:potential} disappears, and the system jumps to the other stable equilibrium. When $h$ is varied back, the jump occurs at a different parameter value. This creates a hysteresis loop.

In MoE terms, the mechanism is as follows. If the router reinforces the already selected expert faster than regularization returns the scores to balance, then the symmetric state loses stability. Small asymmetries in the data or initialization select one of the two branches, after which recovering balance requires crossing a fold threshold. In this model, that threshold property is what makes load collapse persistent.

\begin{remark}
In this interpretation, a load-balancing loss may act in several ways: it can decrease the effective feedback parameter $a$, increase the regularizing coefficient $\gamma$, increase the effective temperature $T$, or add negative feedback depending on current load. All of these mechanisms move the system away from the multistable region.
\end{remark}

\section{Numerical Experiments with Batch Routing}

The analytical part describes mean-field dynamics. To connect it with MoE routing, we next consider the original discrete rule \eqref{eq:discrete-update} numerically in batch form. At each step a batch of $B$ tokens is routed; the number of tokens sent to the first expert is denoted $N_1$, and $N_2=B-N_1$. The observed quantity in all experiments is the empirical load imbalance
\[
    \widehat u = \frac{N_1-N_2}{B}.
\]
The figures below show expert load directly in the stochastic router.

\subsection{A Small Trainable MoE Model with an Input-Dependent Router}

To test the connection between the reduced picture and trainable routing, consider a small MoE model on a synthetic regression problem with two regimes. Let $x\in[-1,1]$ and define the target
\[
    y(x)
    =
    \begin{cases}
        -1-0.7x+0.05\sin(8\pi x), & x<0,\\
        1+0.7x+0.05\sin(8\pi x), & x\ge 0.
    \end{cases}
\]
The two experts are affine regressors, and the router depends on the input:
\[
    p_1(x)=\sigma\left(\frac{\alpha x+\beta+h}{T}\right),
    \qquad
    p_2(x)=1-p_1(x).
\]
Here $\alpha$ and $\beta$ are trainable router parameters, while $h$ is an external bias modeling data skew or an architectural preference for one expert. The model prediction is
\[
    \widehat y(x)=p_1(x) f_1(x)+(1-p_1(x))f_2(x),
\]
and the expert parameters together with $(\alpha,\beta)$ are trained by stochastic gradient descent on MSE with a small regularization of router weights.

In Figure \ref{fig:tiny-trainable-moe-input-router}, the parameter $h$ is slowly scanned upward and downward; at each new value training continues from the previous state. The top panel shows the expected load imbalance $\mathbb{E}\widehat u=2\mathbb{E}_x p_1(x)-1$, the middle panel shows the learned decision boundary $x_\ast=-(\beta+h)/\alpha$, and the bottom panel shows MSE. In the central bias region, the model uses both experts and learns a boundary near $x=0$. For sufficiently large $|h|$, the external bias suppresses the input-dependent partition, the router collapses to one expert, and the error increases. The two scan directions correspond to different expert label assignments, an expected consequence of the symmetry of the two-expert model.

\begin{figure}[htbp]
    \centering
    \includegraphics[width=0.78\textwidth]{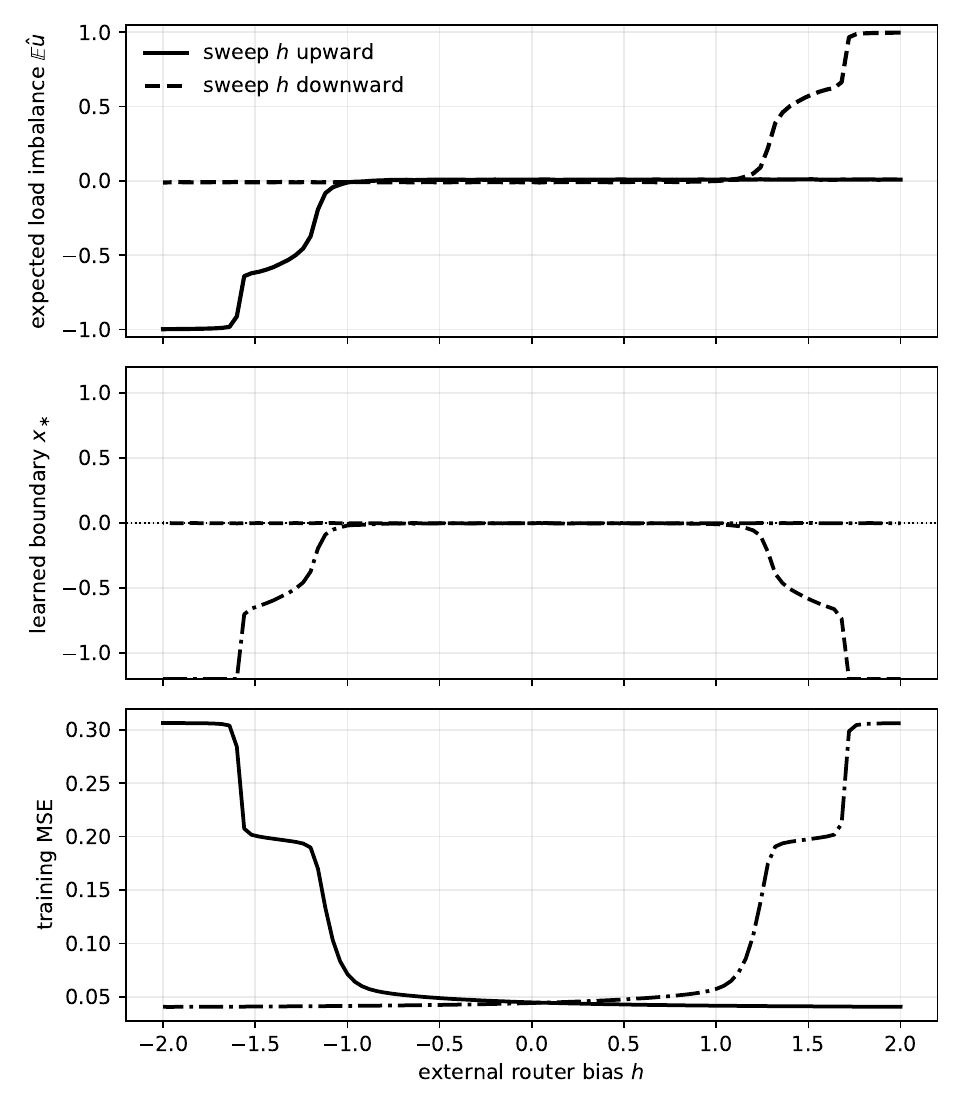}
    \caption{A small trainable MoE model with an input-dependent router. In the central region the model uses both experts and keeps the boundary near $x=0$; for large external bias, the load collapses to one expert.}
    \label{fig:tiny-trainable-moe-input-router}
\end{figure}

\subsection{Pilot PyTorch Experiment with a Hard Top-1 Router}

The previous experiment uses smooth expert mixing and is therefore closer to a soft MoE than to a sparse MoE. To illustrate which parts of the mechanism persist under discrete routing, consider a small PyTorch model \cite{Paszke2019} with hard top-$1$ expert selection. The experts are again affine regressors, and the router has two logits
\[
    z(x)=Wx+(h/2,-h/2).
\]
In the forward pass one expert is selected:
\[
    g(x)=\operatorname{onehot}\arg\max_i z_i(x),
\]
while router training uses a straight-through estimate \cite{Bengio2013}:
\[
    g_{\mathrm{st}}(x)=g(x)+p(x)-\operatorname{stopgrad}p(x),
    \qquad
    p(x)=\operatorname{softmax}(z(x)/T).
\]
The objective is MSE; in some runs we add the standard penalty on router soft importance,
\[
    L_{\mathrm{lb}}
    =
    \lambda_{\mathrm{lb}}
    \sum_{i=1}^2
    \left(
        \frac{1}{B}\sum_{x\in B}p_i(x)-\frac12
    \right)^2.
\]
This experiment is not a proof of the mean-field model. Hard top-$1$ introduces additional effects, especially dead-router regions. Its narrower role is to illustrate whether a similar collapse/balancing mechanism appears in a minimal trainable sparse MoE.

Figure \ref{fig:pytorch-top1-moe-bias-sweep} scans the external bias $h$ slowly upward and downward. With $\lambda_{\mathrm{lb}}=0$, hard load quickly saturates near $\widehat u=\pm 1$, corresponding to effective use of a single expert. In this region MSE increases because one affine expert is forced to approximate both regimes of the target function. With $\lambda_{\mathrm{lb}}=1$, the transition is smoother and full saturation of the load is delayed: balancing does not remove the external bias, but it weakens collapse.

\begin{figure}[htbp]
    \centering
    \includegraphics[width=0.82\textwidth]{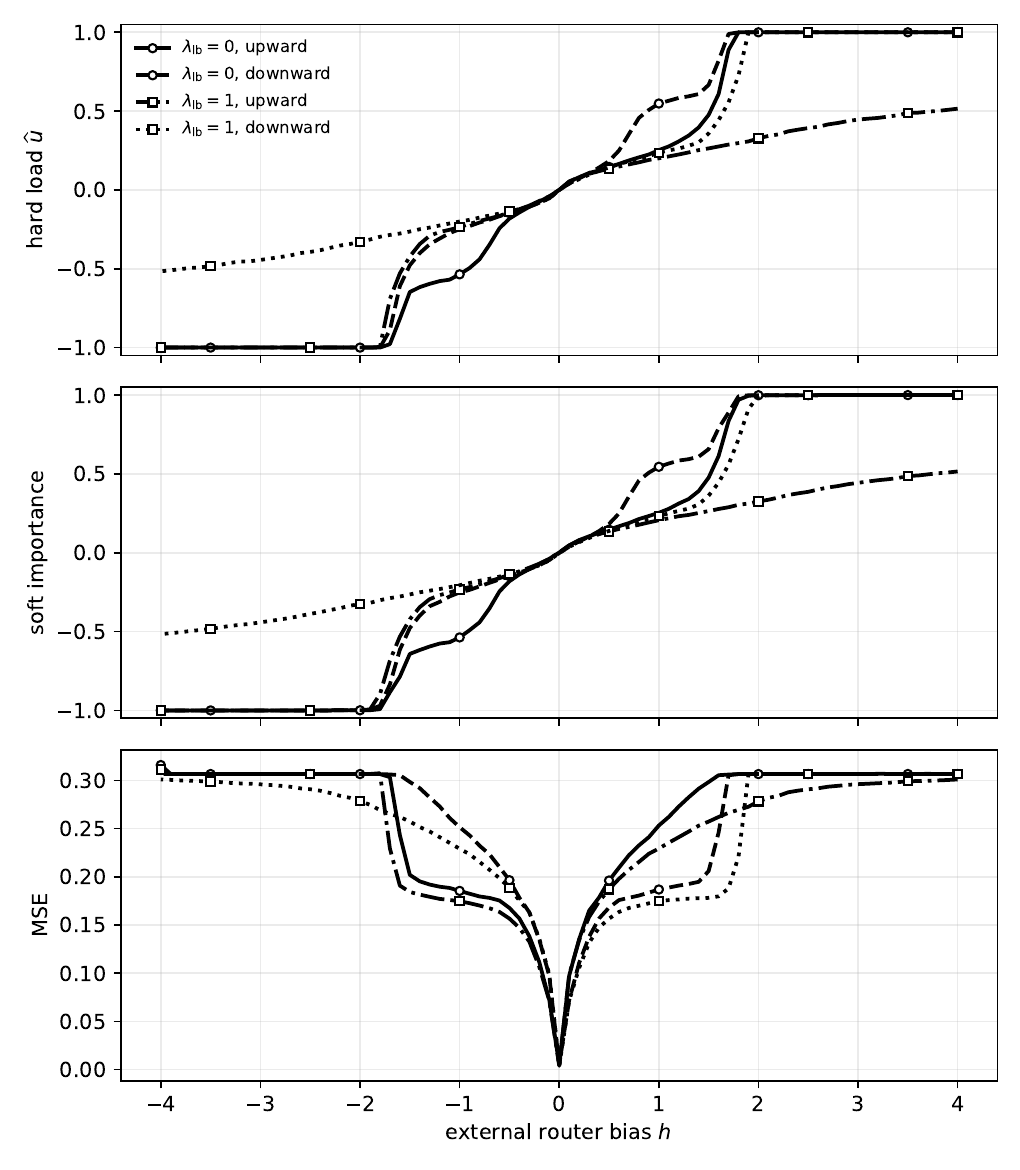}
    \caption{Pilot PyTorch MoE with hard top-$1$ routing. Top: empirical hard expert load. Middle: router soft importance. Bottom: MSE. Solid curves correspond to no load-balancing penalty; dashed curves correspond to $\lambda_{\mathrm{lb}}=1$.}
    \label{fig:pytorch-top1-moe-bias-sweep}
\end{figure}

Figure \ref{fig:pytorch-top1-moe-balance-scan} shows the effect of the balancing loss in the hard-routed model. At fixed external bias $h=2$, the model is trained from random initialization for different $\lambda_{\mathrm{lb}}$. Small penalty values barely change the final hard load: the model remains close to collapse. After a finite threshold, the balancing term substantially decreases $|\widehat u|$, while MSE also decreases. This agrees qualitatively with the reduced picture: negative feedback on load should move the system out of the region of stable imbalance. It is not, however, a test of the fold formula for a discontinuous system. The hard top-$1$ model has its own dead-router regimes, and the STE provides only a heuristic gradient for training.

\begin{figure}[htbp]
    \centering
    \includegraphics[width=0.82\textwidth]{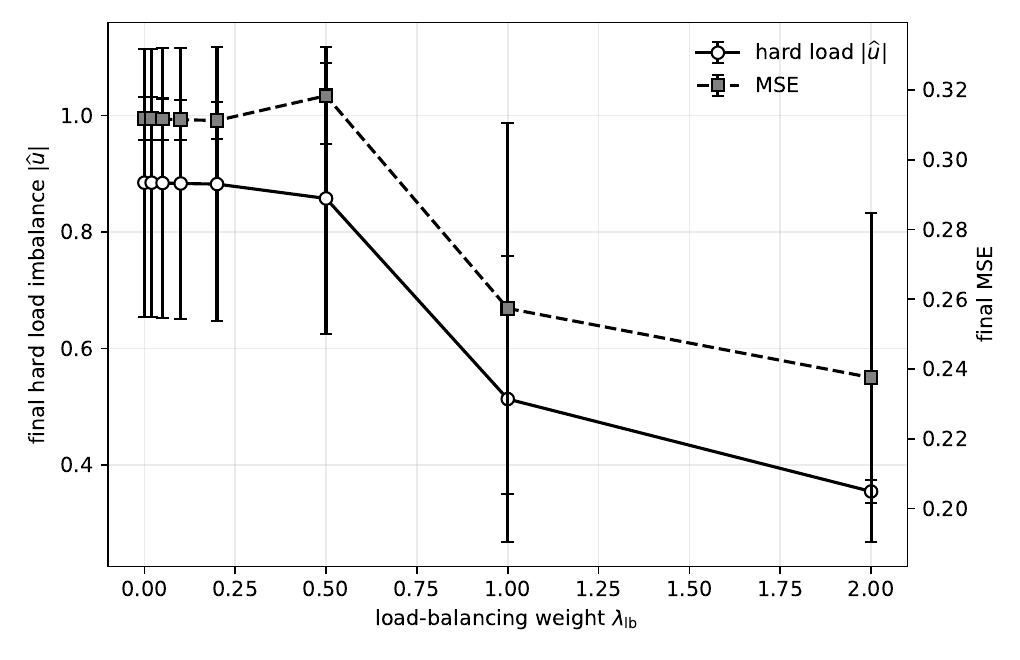}
    \caption{Effect of a load-balancing penalty on a trainable hard top-$1$ MoE at fixed bias $h=2$. Solid curve with circles: final absolute hard load imbalance. Dashed curve with squares: final MSE. Means and standard deviations over independent initializations are shown.}
    \label{fig:pytorch-top1-moe-balance-scan}
\end{figure}

\subsection{Classification Experiment on \texttt{digits}}

As a controlled classification sanity check, consider the small handwritten-digit classification task \texttt{digits} from \texttt{scikit-learn} \cite{Pedregosa2011}. The dataset consists of $8\times 8$ grayscale images from ten classes. We use an MoE with two hard top-$1$ experts; each expert is a small bottleneck MLP
\[
    \mathbb{R}^{64}\to \mathbb{R}^{4}\to \mathbb{R}^{10}.
\]
Limited expert capacity is essential here: if the experts are made sufficiently wide, one expert can solve the task by itself, and collapse hardly affects accuracy. This experiment therefore probes the regime in which specialization between two experts has measurable value.

The router is
\[
    z(x)=Wx+c+(h/2,-h/2),
\]
followed by hard top-$1$ selection with the same straight-through estimate as above. The model is trained with cross-entropy; balancing is implemented by the same penalty on mean router soft importance. Each point is averaged over five independent initializations.

Figure \ref{fig:pytorch-digits-moe-bias-scan} scans the external bias $h$. Without balancing, increasing $h$ increases the test hard-load imbalance and simultaneously decreases test accuracy. With balancing $\lambda_{\mathrm{lb}}=1$, the load remains closer to symmetric and accuracy does not show the same drop. This illustrates the applied meaning of the reduced picture: a persistent router skew may be associated with degraded performance, and negative feedback on load can reduce both effects.

\begin{figure}[htbp]
    \centering
    \includegraphics[width=0.82\textwidth]{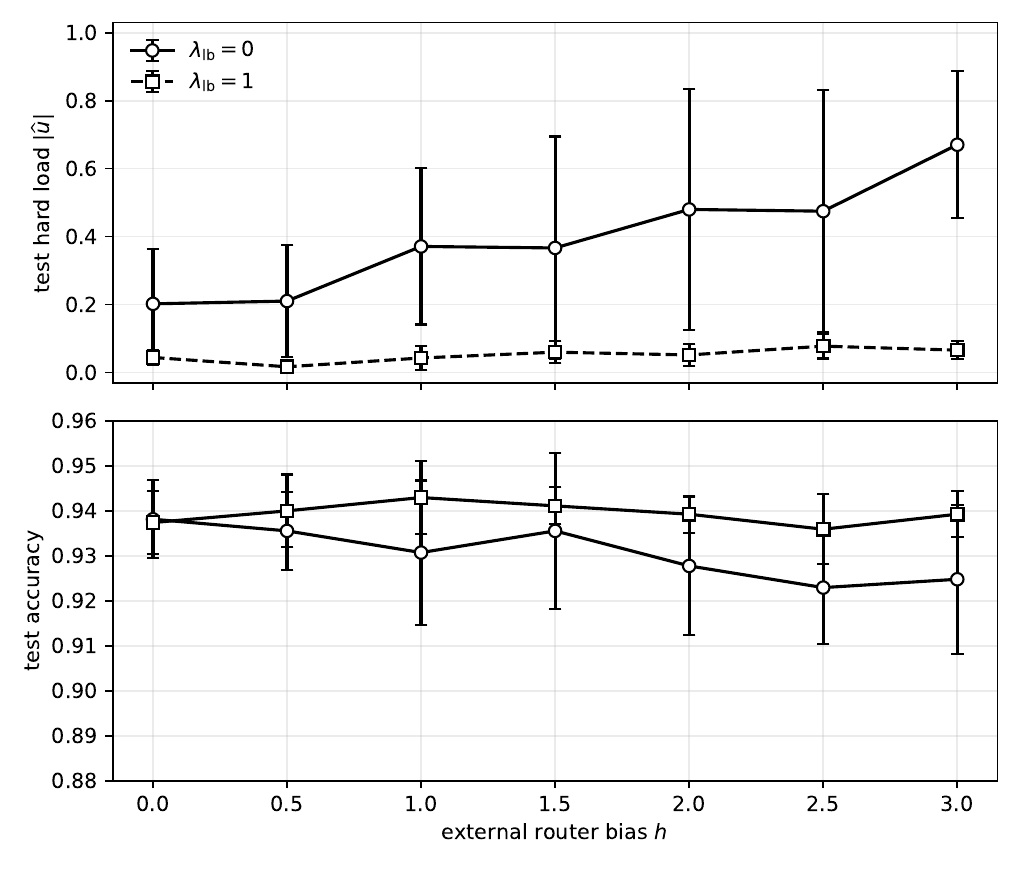}
    \caption{Hard top-$1$ MoE on \texttt{digits}. Top: absolute test load imbalance. Bottom: test accuracy. Without balancing, external bias increases collapse and degrades performance; a load-balancing penalty keeps the load closer to symmetry.}
    \label{fig:pytorch-digits-moe-bias-scan}
\end{figure}

Figure \ref{fig:pytorch-digits-moe-balance-scan} fixes $h=2$ and varies $\lambda_{\mathrm{lb}}$. The model transitions from a regime with noticeable imbalance to nearly balanced expert usage. Accuracy is not monotone at every individual point, but the overall trend is stable: decreasing hard-load imbalance is accompanied by recovery of performance. This result is important for interpretation, but it should be understood as a small-scale sanity check. Here balancing equalizes expert counters and improves generalization of a small MoE model in a limited-capacity regime.

\begin{figure}[htbp]
    \centering
    \includegraphics[width=0.82\textwidth]{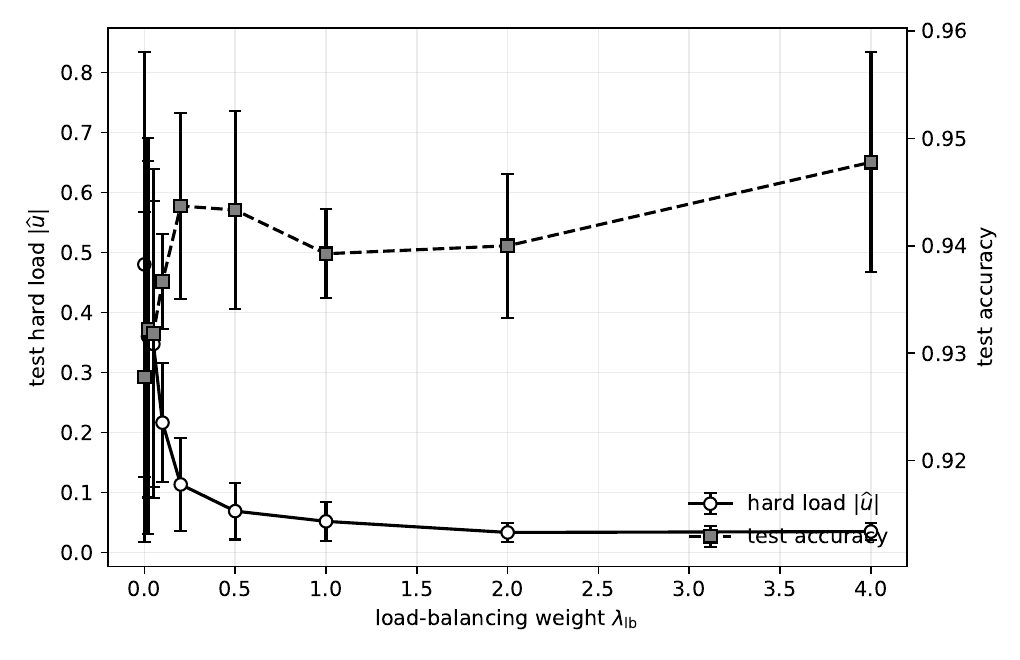}
    \caption{Scan over the load-balancing strength on \texttt{digits} at fixed bias $h=2$. As $\lambda_{\mathrm{lb}}$ increases, hard-load imbalance decreases and test accuracy recovers on average.}
    \label{fig:pytorch-digits-moe-balance-scan}
\end{figure}

\subsection{Mean-Field Description of Empirical Load}

We first test whether the limiting ODE describes the averaged dynamics of the batch router. Figure \ref{fig:moe-stochastic-mean-field} compares the two at $a=3$, $h=0.08$, $\gamma=T=1$, $B=512$, and $\eta=0.002$. The solid curve is the mean empirical load imbalance over $40$ independent runs of the discrete router; the dashed curve is the mean-field ODE solution mapped to load by $u=\tanh(y/2T)$.

\begin{figure}[htbp]
    \centering
    \includegraphics[width=0.78\textwidth]{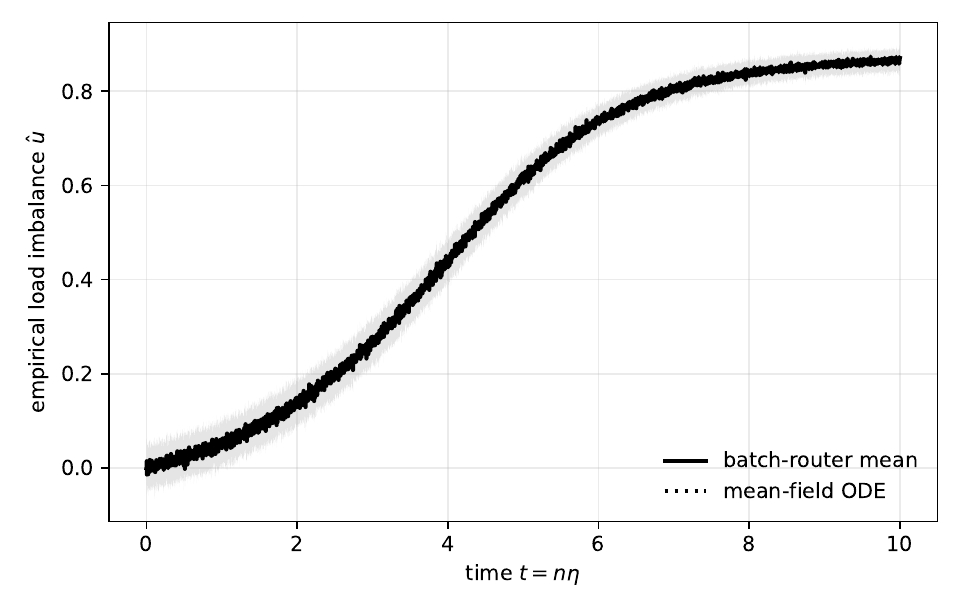}
    \caption{Batch token routing and the mean-field limit. The shaded band shows one standard deviation of the empirical imbalance $\widehat u$ over independent runs.}
    \label{fig:moe-stochastic-mean-field}
\end{figure}

\subsection{Hysteresis in Expert Load}

The next experiment fixes $a=4$, $\gamma=T=1$ and slowly varies the external skew $h=b_1-b_2$. For each value of $h$, several steps of batch routing are performed and the actual expert load is averaged. The same sweep is then repeated in the reverse direction.

Figure \ref{fig:moe-hysteresis-load} shows that at the same value of $h$ the router may be in different load regimes depending on its history. This is the MoE interpretation of hysteresis: if one expert has already obtained a persistent advantage, a small change in skew does not return the system to balanced load. The transition occurs only when the fold threshold is reached and the current stable routing branch disappears.

\begin{figure}[htbp]
    \centering
    \includegraphics[width=0.78\textwidth]{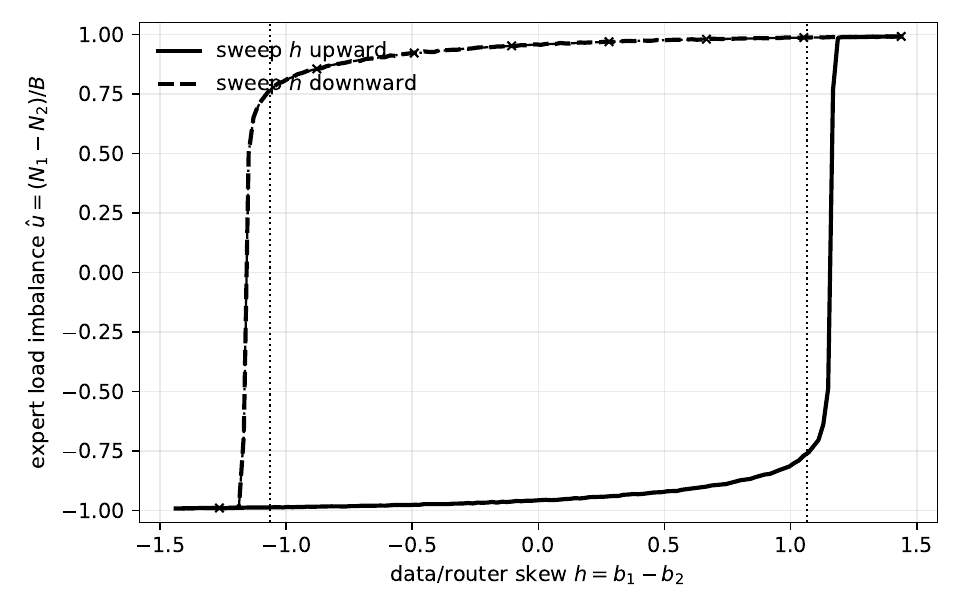}
    \caption{Hysteresis of empirical load in the two-expert batch router. The vertical axis shows $\widehat u=(N_1-N_2)/B$, not the hidden variable $y$. Dashed vertical lines are the fold thresholds predicted by the mean-field model.}
    \label{fig:moe-hysteresis-load}
\end{figure}

\subsection{Temperature, Regularization, and the Collapse Region}

Finally, consider the symmetric case $h=0$ and scan $T$ and $\gamma$ at fixed reinforcement strength $a=3$. For each parameter pair, we run stochastic batch-routing dynamics with a small random initial asymmetry. The fill tone in Figure \ref{fig:moe-collapse-map} shows the average final value of $|\widehat u|$. Dark regions correspond to nearly balanced expert usage, and light regions to persistent load skew.

The dashed line shows the analytical threshold $a=2\gamma T$. Agreement with the numerical transition means that softmax temperature and regularizing forgetting act as controllable stability parameters for routing: increasing $T$ or $\gamma$ moves the system from the load-collapse region to the balanced-usage region.

\begin{figure}[htbp]
    \centering
    \includegraphics[width=0.78\textwidth]{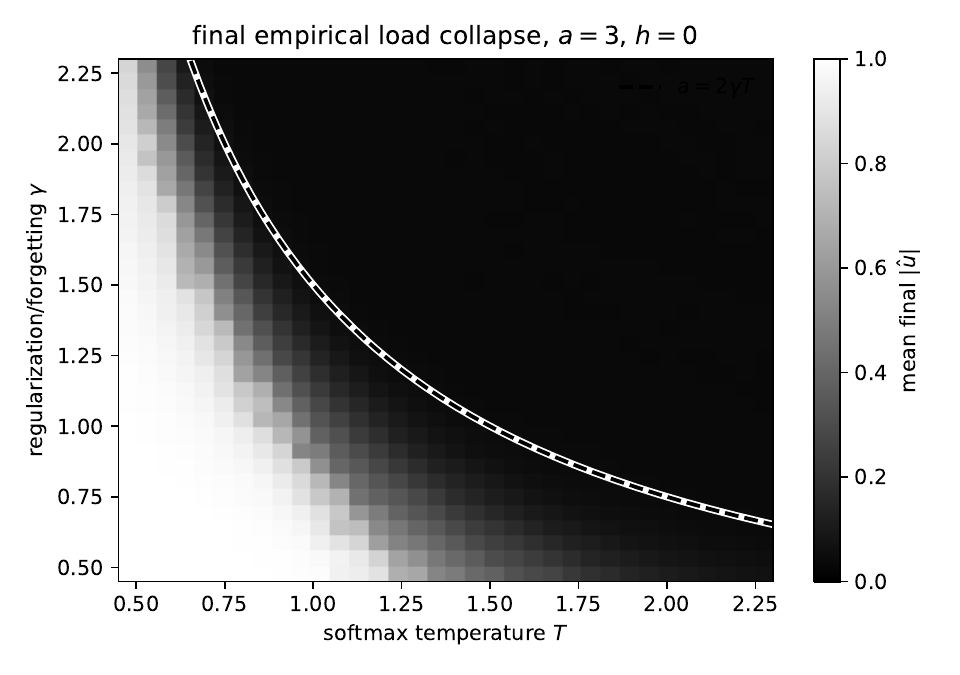}
    \caption{Final empirical load imbalance at $h=0$ as a function of softmax temperature and regularization. The dashed line is the mean-field threshold $a=2\gamma T$.}
    \label{fig:moe-collapse-map}
\end{figure}

\subsection{Quantitative Tests of the Predictions}

The previous experiments show the mechanisms qualitatively. We now test three numerical predictions of the model more directly.

The first prediction is the critical condition
\[
    a=2\gamma T.
\]
At fixed $a=3$ and for different $\gamma$, we measure a finite-time onset of collapse: for each temperature we run batch-routing dynamics at $h=0$ and record the final mean $|\widehat u|$. The estimated threshold is the value of $T$ at which the final imbalance reaches $|\widehat u|=0.1$. Figure \ref{fig:validation-critical-temperature} compares this estimate with the formula $T_c=a/(2\gamma)$. The measured finite-time onset lies below the analytical critical temperature, as expected: near the threshold the linear growth of the unstable mode is slow, so at finite time and finite batch size the observed imbalance appears later. Nevertheless, the dependence on $\gamma$ reproduces the predicted scale.

\begin{figure}[htbp]
    \centering
    \includegraphics[width=0.72\textwidth]{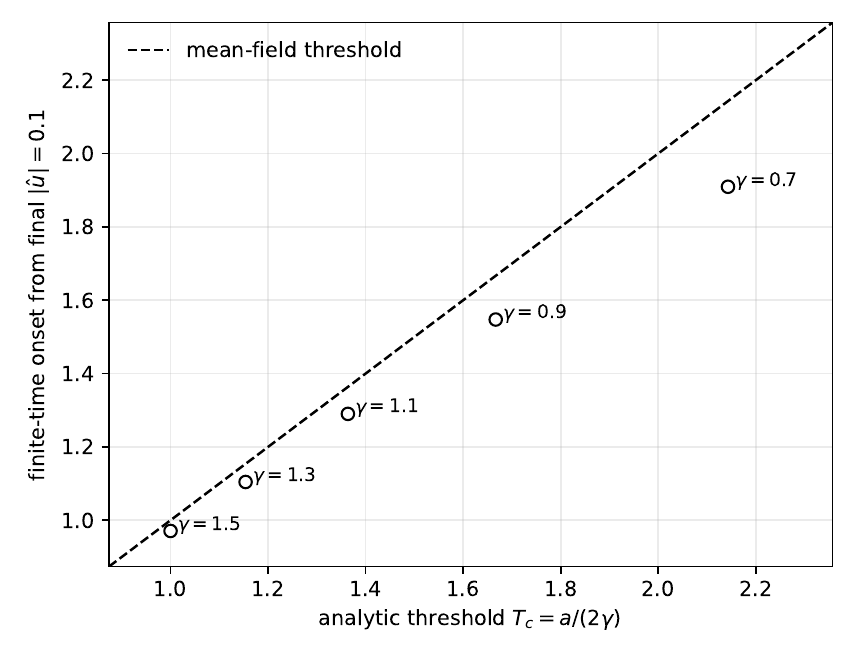}
    \caption{Test of the temperature threshold. Points show the finite-time onset of load collapse in batch-routing simulation; the dashed line is the analytical condition $T_c=a/(2\gamma)$.}
    \label{fig:validation-critical-temperature}
\end{figure}

The second prediction concerns the width of the hysteresis loop. For each $a>2\gamma T$, two slow sweeps over $h$ are performed, one upward and one downward. The switching threshold is defined as the point at which the empirical imbalance $\widehat u$ changes sign. The measured width $\Delta h$ is compared with
\[
    \Delta h_{\mathrm{mf}}(a)
    =
    2H(a),
\]
where $H(a)$ is defined in \eqref{eq:H}. Figure \ref{fig:validation-hysteresis-width} shows that batch-routing dynamics reproduces the growth of the loop width as the positive feedback strength increases.

\begin{figure}[htbp]
    \centering
    \includegraphics[width=0.78\textwidth]{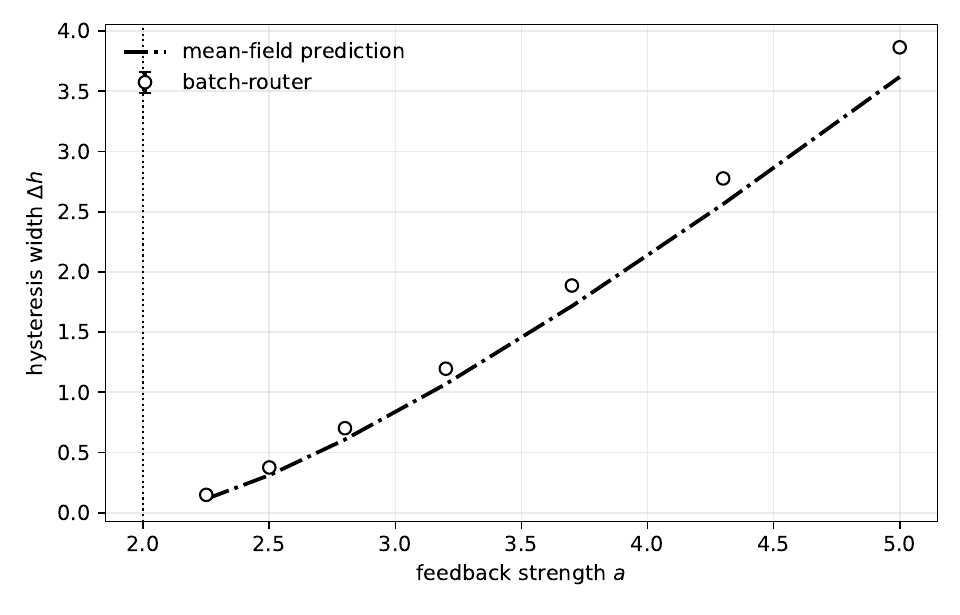}
    \caption{Hysteresis width as a function of positive feedback strength $a$. Points are batch-routing simulations; the dashed line is the mean-field prediction $2H(a)$.}
    \label{fig:validation-hysteresis-width}
\end{figure}

The third prediction concerns balancing mechanisms. Add negative feedback on load to the discrete rule:
\[
    r_i^{n+1}
    =
    r_i^n
    +
    \eta
    \left(
        a\ell_i^n
        -
        \rho(\ell_i^n-\tfrac12)
        -
        \gamma r_i^n
        +
        b_i
    \right),
\]
where $\ell_i^n=N_i^n/B$ is the fraction of the batch sent to expert $i$. In the difference variable this replaces the effective positive feedback strength $a$ by $a_{\mathrm{eff}}=a-\rho$. The model therefore predicts the disappearance of hysteresis when
\[
    a-\rho \le 2\gamma T.
\]
Figure \ref{fig:validation-balancing-feedback} is consistent with this behavior: as $\rho$ increases, the measured loop width decreases and becomes small near the predicted threshold.

\begin{figure}[htbp]
    \centering
    \includegraphics[width=0.78\textwidth]{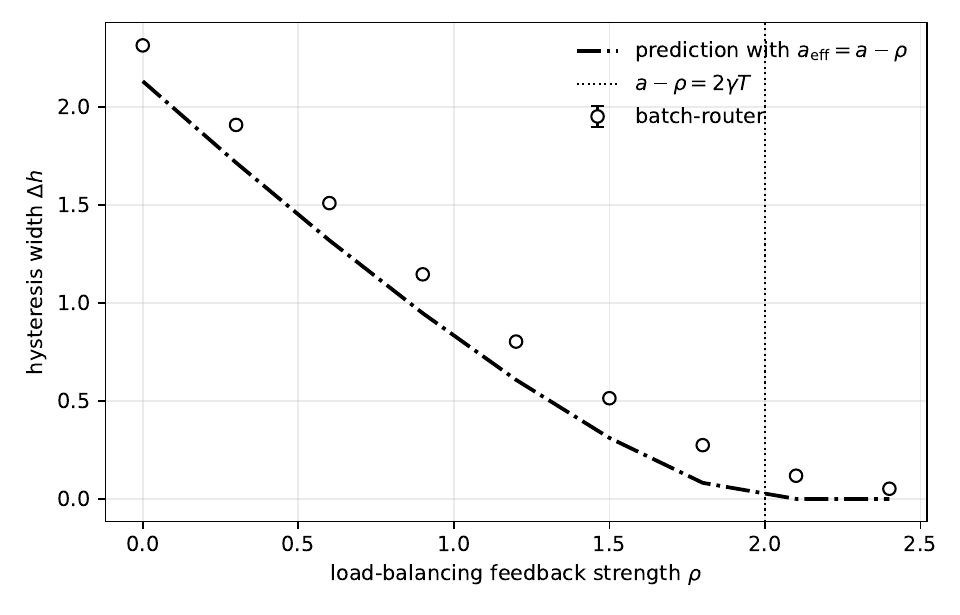}
    \caption{Suppression of hysteresis by negative feedback on load. The dashed line is the prediction with effective parameter $a_{\mathrm{eff}}=a-\rho$; the vertical line marks $a-\rho=2\gamma T$.}
    \label{fig:validation-balancing-feedback}
\end{figure}

\section{Limitations}

The proposed system is deliberately minimal. The analytical part describes only two experts and continuous softmax routing. The PyTorch experiments partially relax this restriction because they use hard top-$1$ selection and include both synthetic regression and a small classification task. In realistic MoE systems, several additional effects may play an essential role:
\begin{itemize}
    \item discrete top-$k$ selection and router noise;
    \item finite batch size and stochasticity of load estimates;
    \item explicit dependence of expert quality on the data distribution;
    \item interactions among more than two experts;
    \item additional balancing terms in the loss.
\end{itemize}
The results should therefore be understood as an analysis of a local mechanism, not as a theorem about arbitrary MoE models. Nevertheless, a two-expert reduction arises naturally when studying competition between two dominant experts or one contrast mode in a multidimensional system. The dead-router behavior observed in the hard top-$1$ experiment also indicates a limitation of the smooth model: a discrete argmax may create practically irreversible dead-router regimes that do not reduce to an ordinary fold of a smooth vector field.

\section{Conclusion}

We introduced a minimal model of an adaptive softmax router:
\[
    \dot y
    =
    a\tanh\frac{y}{2T}
    -
    \gamma y
    +
    h.
\]
For this model we completely described the bifurcation picture. In the symmetric case we proved a supercritical pitchfork bifurcation at $a=2\gamma T$. Under symmetry breaking we derived an exact parametric description of the fold set and showed that locally it has the structure of a cusp catastrophe. The model exhibits hysteresis and abrupt transitions between stable routing regimes.

The main conclusion is that positive feedback in a softmax router can by itself create a region of multistability. This gives a rigorous low-dimensional model on which stability thresholds and load-balancing mechanisms can be analyzed.

Pilot PyTorch experiments with hard top-$1$ routing show that the same qualitative mechanism persists in small trainable sparse MoE models: external bias can lead to load collapse and increased error, while a load-balancing penalty can partially restore the use of both experts. On \texttt{digits}, this effect is visible in test accuracy when expert capacity is limited. At the same time, these experiments reveal a difference from the smooth theory: hard routing creates dead-router regions in which recovering balance is substantially harder.

\section*{Acknowledgements and AI Assistance}

AI-based tools were used for editing and debugging the English text and for assistance in writing scripts used in the numerical experiments. The author is responsible for the mathematical content, numerical interpretation, and final form of the manuscript.

\end{document}